\documentclass[a4paper,12pt]{amsart}

\usepackage{amsmath}
\usepackage{amssymb}
\usepackage{amsthm}
\usepackage{graphicx}
\usepackage{amscd}
\usepackage{xypic}
\usepackage{geometry}

\newtheorem{theorem}{Theorem}[section]
\newtheorem{proposition}[theorem]{Proposition}
\newtheorem{lemma}[theorem]{Lemma}
\newtheorem{corollary}[theorem]{Corollary}
\newtheorem{definition}[theorem]{Definition}
\newtheorem{remark}[theorem]{Remark}
\newtheorem{example}[theorem]{Example}

\newcommand{\gauss}{ {}_2 { \rm F}_1 }

\newcommand{\appellett}{ { \rm F}_1 }

\newcommand{\m}{\operatorname{m}}

\newcommand{\Li}{\operatorname{Li}}
\newcommand{\K}{\operatorname{K}}

\newcommand{\Log}{\operatorname{Log}}

\newcommand{\C}{\mathbb{C}}
\newcommand{\Cont}{\mathcal{C}}
\newcommand{\Proj}{\mathbb{P}}
\newcommand{\R}{\mathbb{R}}
\newcommand{\Z}{\mathbb{Z}}
\newcommand{\Af}{\mathcal{A}_f}

\newcommand{\Length}{\operatorname{Length}}
\newcommand{\ddx}{\frac{\partial}{\partial x}}
\newcommand{\ddy}{\frac{\partial}{\partial y}}

\newcommand{\ddr}{\frac{d}{dr}}
\input xy
\xyoption{all}

\begin{document}
\title{An explicit calculation of the Ronkin function}
\author{Johannes Lundqvist}
\date{}
\address{Department of Mathematics\\
Stockholm University\\
SE-106 91 Stockholm\\ 
Sweden\\}
\email{johannes@math.su.se}
\maketitle

\begin{abstract}
We calculate the second order derivatives of the Ronkin function in the case of an affine linear polynomial in three variables and give an expression of them in terms of complete elliptic integrals and hypergeometric functions.
This gives a semi-explicit expression of the associated Monge-Amp\`ere measure, the Ronkin measure.
\end{abstract}

\section{Introduction}\label{intro}

Amoebas are certain projections of sets in $\C^n$ to $\R^n$ that are connected to several areas in mathematics such as complex analysis, tropical geometry, real algebraic geometry, special functions and combinatorics to name a few. Amoebas were first defined by Gelfand, Kapranov and Zelevinsky in \cite{GKZ} and these objects were later studied by several other authors like Mikhalkin, 
Passare, Rullg\aa rd, and Tsikh. The Ronkin function of a polynomial is closely connected to the amoeba of that polynomial. The main result in this paper is an explicit calculation of the second order derivatives of the Ronkin function in the case of an affine linear polynomial $f$ in three dimensions, thus giving an explicit expression of the so-called Ronkin measure associated to $f$.
\\
\\
Assume that $f$ is a Laurent polynomial in $n$ variables over $\C$. That is,
\begin{equation*}
f(z)=\sum_{\alpha\in A}a_{\alpha}z^{\alpha}
\end{equation*}
for some finite set $A\subset \Z^n$. The convex hull in $\R^n$ of the points $\alpha\in A$ for which $a_{\alpha}\neq 0$ is called the Newton polytope of $f$ and is denoted by $\Delta_f$.
\begin{definition}{\rm (Gelfand, Kapranov, Zelevinsky)}\label{amoeba}
Let $f(z)$ be a Laurent polynomial in $n$ variables over $\mathbb{C}$. The {\rm amoeba}, $\Af$, of $f$ is the image of $f^{-1}(0)$ under the map $\Log: (\C^*)^n\to\R^n$ defined by
\begin{equation*}
\Log(z_1,\ldots,z_n)=(\log|z_1|,\ldots,\log|z_n|).
\end{equation*}
The {\rm compactified amoeba}
of $f$, denoted by $\bar{\Af}$, is the closure of the image of $f^{-1}(0)$ under the map $\nu:(\C^*)^n\to\Delta_f$ defined by
\begin{equation*}
\nu(z_1,\ldots,z_n)=\frac{\sum_{\alpha\in A}|z^{\alpha}|\cdot\alpha}{\sum_{\alpha\in A}|z^{\alpha}|}.
\end{equation*}
\end{definition}
\noindent
We get the following commutative diagram
\begin{center}
$\xymatrix{
(\mathbb{C}^*)^n \ar[r]^{\Log} \ar[dr]^{\nu} &\mathbb{R}^n \ar[d]^{\gamma}\\
&\text{int}(\Delta_f),
}$
\end{center}
where 
\begin{equation*}
\gamma(x) = \frac{\sum_{\alpha \in A}e^{\left< \alpha,x\right> }\cdot \alpha}{\sum_{\alpha\in A}e^{\left<\alpha,x \right>}} 
\end{equation*}
is a diffeomorphism.
\smallskip

\noindent
The connected components of the complement of the amoeba of a Laurent polynomial are convex, see \cite{GKZ}. Moreover, the number of connected complement components is at least equal to the number of vertices in $\Delta_f\cap\mathbb{Z}^n$ and at most equal to the number of points in $\Delta_f\cap\mathbb{Z}^n$. That is, there exists an injective function from the set of connected components of $\R^n \setminus \Af$ to $\Delta\cap\mathbb{Z}^n$. 

Such an injective function can be constructed with the Ronkin function
\begin{equation*}
N_f(x)=\left( \frac{1}{2\pi i}\right) ^n\int_{Log^{-1}(x)}
\log|f(z)|\frac{dz}{z},\quad x\in\R^n.
\end{equation*}
It is a multivariate version of the mean value term in Jensen's formula and was first studied by Ronkin, see \cite{Ronkin}.
The Ronkin function of a product of two polynomials is obviously the sum of the Ronkin function of those two polynomials. It is also easy to see that the Ronkin function of a monomial $az^{\alpha}\in\C[z_1,\ldots ,z_n]$ is an affine linear polynomial in $\R[x_1,\ldots,x_n]$, i.e., if $f(z) = az_1^{\alpha_1}\ldots z_n^{\alpha_n}$ then
\begin{equation*}
N_f = \log|a| + \alpha_1x_1 + \alpha_2x_2 +\ldots + \alpha_nx_n.
\end{equation*}
The function $N_f$ is convex on $\R^n$ and it is affine linear on an open connected set $\Omega\subset\R^n$ if and only if $\Omega\subset\R^n\setminus\Af$.
In fact the gradient of $N_f$ at a point outside the amoeba is a point in $\Delta\cap\mathbb{Z}^n$, and thus the Ronkin function gives a mapping from the set of complement components to the set of points in $A$. This mapping was proved to be injective in \cite{FPT}. Moreover, it is easy to see that the amoeba always has components corresponding to the vertices in $\Delta\cap\mathbb{Z}^n$, and thus we get the inequalities on the number of complement components above.

\begin{example}
{\rm
Let 
\[
f(z)=a_0+a_1z+a_2z^2+\cdots + a_nz^n=(z-b_1)\cdots(z-b_n),
\]
where $a_0\neq 0$ and $b_1 \leq b_2<\ldots \leq b_n$. Then for $x$ such that $b_m<e^x<b_{m+1}$ we get
\begin{eqnarray*}
N_f(x) &=& \int_0^{2\pi}\log|f(e^{x+i\phi})|d\phi=\log|a_0| + \sum_{k=1}^m\log\left(\frac{e^x}{|b_k|}\right) = \\
&=& \log|a_0|-\sum_{k=1}^m\log|b_k| + mx
\end{eqnarray*}
by Jensen's formula, and we see that $N_f$ is a convex piecewise affine linear function, singular at $\log|b_k|, \quad k=1,\ldots, n$.
}
\end{example}

\begin{definition}
Let $f$ be a Laurent polynomial. The real Monge-Amp\`ere measure of $N_f$ is called the {\rm Ronkin measure} associated to $f$ and is denoted by $\mu_f$.
\end{definition}

\noindent
Since $N_f$ is affine linear outside the amoeba of $f$, the measure $\mu_f$ has its support on the amoeba.
Moreover, Passare and Rullg\aa rd proved that $\mu_f$ has finite total mass and that the total mass equals the volume of the Newton polytope of $f$, see \cite{PassRull}.
They also proved that 
\begin{equation}\label{olikhet}
\mu_f\geq\frac{\lambda}{\pi^2}
\end{equation}
on the amoeba of $f$, where $\lambda$ is the Lebesgue measure. From this estimate the following theorem follows immediately.
\begin{theorem}{\rm (Passare, Rullg\aa rd)}\label{olik}
In the two variable case the area of the amoeba of $f$ is bounded by $\pi^2$ times the area of the Newton polytope of $f$.
\end{theorem}

\noindent
In higher dimension, very little is known about the Ronkin measure.
There is no hope of finding an analogue of Theorem~\ref{olik} in more than two variables because in that situation the volume of the amoeba is almost always infinite, see Example~\ref{ex} below.
There might still be an inequality like the one in \eqref{olikhet} but with $1/\pi^2$ replaced by a function. In this paper we investigate the Ronkin measure of the affine linear polynomial $f=1+z+w+t$, which should be the easiest possible three variable example. There are some known explicit formulas for the closely connected Mahler measure, see Section~\ref{Mahler}, in this case but they only give information about the Ronkin measure on special curves on the amoeba.
In particular, there is no known explicit expression for the Ronkin function, which indicates that it might be considerably more complicated than in the two variable case. On the other hand, it seems that the Ronkin measure is easier to calculate than the Ronkin function itself.
The main result in this paper is that the Ronkin measure of an affine linear polynomial in three variables can be explicitly described in terms of complete elliptic integrals or hypergeometric functions.


\section{Hyperplane amoebas}\label{hyper}

The Newton polytope of a hyperplane amoeba in $n$ variables has $n+1$ integer points and all of them are vertices. This implies that the amoeba has exactly $n+1$ complement components according to the results discussed directly below Definition \ref{amoeba}. Moreover, since there is only one subdivision of the Newton polytope, the trivial one, we know that the amoeba is solid, i.e., has no bounded 
complement components. The compactified hyperplane amoebas turn out to be particularly easy to express.
They are in fact polytopes.

\begin{proposition}{\rm (Forsberg, Passare, Tsikh)}\label{FPT}
Let $f$ be the affine linear polynomial $a_0+a_1z_1+a_2z_2+\ldots+a_nz_n$ and assume that $|a_j|+|a_k|\neq 0 $ for all $j$ and $k$. Then $\bar{\Af}$ is the convex hull of the points $v_{jk}=(t_1,\ldots,t_n)$, $j\neq k$, where either
\begin{eqnarray*}
&&t_j = \frac{|a_0|}{|a_j|+|a_0|}, \quad t_l = 0 \quad \text{for}\quad l\neq j\quad, \quad\text{or} \\
&&t_j = \frac{|a_k|}{|a_j|+|a_k|}, \quad t_k = \frac{|a_j|}{|a_k|+|a_j|}, \quad t_l=0 \quad \text{for } l\neq j,k.
\end{eqnarray*}
\end{proposition}

\begin{figure}[h]
\begin{center}
\includegraphics[height=3cm]{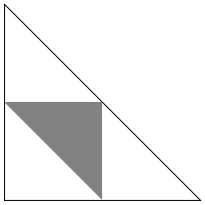}
\qquad
\includegraphics[height=3cm]{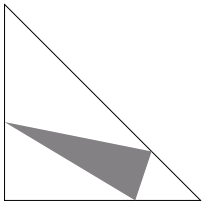}
\caption{The compactified amoebas of $f(z,w)=1+z+w$ and \newline $f(z,w)=2+z+3w$}
\label{triangel}
\end{center}
\end{figure}

In Section \ref{intro} we saw that the area of an amoeba in two variables is finite. That is not true in higher dimension as we see in the example below. 

\begin{example}\label{ex}
{\rm
Let $f=1+z+w+t$.
It follows from \cite[Theorem~1]{Pass} that the 
corner set of $\operatorname{max}(0,x,y,u)$ is included in the amoeba. This tropical hypersurface is called the spine of the amoeba.
In particular, the amoeba contain the ray $(0,0,t)$ for $t\in[-\infty,0]$. Actually a whole cylinder containing that ray is contained in the amoeba. This can be seen in the following way.
Consider the annulus 
\begin{equation*}
\mathcal{U} = \{ 1+r_1e^{i\varphi}+r_2e^{i\theta}; \varphi,\theta\in[0,2\pi], \frac{2}{3}\leq r_1,r_2\leq\frac{4}{3} \}.
\end{equation*}
If $C$ is a circle with center at the origin and with radius $r\leq 1$ then it is obvious that 
$C\cap\mathcal{U}\neq \emptyset$.
This means that a point $(x,y,u\in\R^3)$ lies in the amoeba of $f$ if $x,y\in[\log|2/3|,\log|4/3|]$ and $u\in(-\infty,0]$ thus the amoeba of $f$ contains a set that obviously has infinite volume.
}
\end{example}

The affine linear polynomials in two variables define so-called Harnack curves and the Ronkin measures associated to such polynomials are known to have constant density 
$1/\pi^2$ on the amoeba, \cite{MikRull}. In this case we get such an easy expression of the partial derivatives of $N_f$  that the fact that $\mu_f$ has constant density $1/\pi^2$ is easy to verify directly. In the case of three variables this kind of calculation is harder and will be done in Section \ref{Ronkin}.
By a change of variables we get the lemma below. It will simplify some of the calculations because it reduces the problem to the case where all the coefficients are equal to $1$.
\begin{lemma}\label{lattlemma}
If
\begin{align*}
f(z) &= 1+z_1+\ldots + z_n\quad \text{and}\quad
f_a(z) = 1+a_1z_1+\ldots +a_nz_n,
\end{align*}
then
\begin{equation*}
N_{f_a}(x_1,\ldots, x_n) = N_{f}(x_1+\log|a_1|,\ldots, x_n + \log|a_n|).
\end{equation*}
\end{lemma}
\noindent

An important subset of the amoeba is the so-called contour and it play an important role in the results in this paper.

\begin{definition}
The set of critical values of the mapping $\Log$ restricted to $f^{-1}(0)$ is called the {\rm contour} of $\Af$ and is denoted by $\mathcal{C}$. 
\end{definition}
\noindent
The contour is a real analytic hypersurface of $\R^n$ and the boundary of the amoeba is always included in the contour. 
The following theorem gives a nice description of the contour.

\begin{proposition}{\rm (Mikhalkin)}\label{Mik}
Let $f$ be a Laurent polynomial. The critical points of the map $\Log$ are exactly the ones that are mapped to 
the real subspace $\R\mathbb{P}^{n-1}\subset\C\mathbb{P}^{n-1}$ under the logarithmic Gauss map. That is,
\begin{displaymath}
\Cont = \Log(\gamma^{-1}(\R\Proj^{n-1})).
\end{displaymath}
\end{proposition}
where $\gamma:f^{-1}(0)\to\C\Proj^{n-1}$ is 
\begin{equation*}
\gamma(z_1,\ldots z_n)=[z_1\frac{\partial}{\partial z_1} f(z_1,\ldots, z_n):\ldots:z_n\frac{\partial}{\partial z_n} f(z_1,\ldots, z_n)].
\end{equation*}
A proof can be found in \cite{Mikhalkin}.

The contour of the hyperplane amoeba in three variables is easy to describe and it subdivides the amoeba into eight parts. Proposition \ref{Mik}
gives that the contour for the amoeba of $f=1+z+w+t$ is given by the set of points $(x,y,u)\in\R^3$ that satisfy one of the equalities
\begin{align*}
&1 + e^x = e^y + e^u,\quad 1 + e^y = e^x + e^u, \quad 1 + e^u = e^x + e^y\\
&e^x = 1 + e^y + e^u,\quad e^y = 1 + e^x + e^y,\quad e^u = 1 + e^x + e^y\\
&1 = e^x + e^y + e^u.
\end{align*}
Note that the equalities with only one term on the left hand side 
are parts of the boundary of $\Af$. Indeed, if  for example 
$
1 > e^x + e^y + e^u,
$
then there cannot exist any angles $\varphi_1,\varphi_2,\varphi_3$ such that 
\[
1 + e^{x+i\varphi_1} + e^{y+i\varphi_2} + e^{u+i\varphi_3}=0,
\]
and hence $(x,y,u)\notin\Af$.

\begin{corollary}
Let $f(z,w,t)=1+z+w+t$. The compactified amoeba of $f$ is an octahedron and the contour divides it into eight convex chambers. The part of the contour that is not on the boundary is the union of the three squares naturally defined by the octahedron. See figure \ref{Contour3var}.
\end{corollary}

\begin{proof}
The first part of the corollary is proved by applying Proposition~\ref{FPT}. 
Consider the points on the contour that satisfy 
$1+e^x=e^y+e^y$. These points are mapped to the compactified amoeba by the map $\gamma$ in Definition \ref{amoeba} to points 
\begin{displaymath}
\frac{(t,s,1+t-s)}{2(1+t)}.
\end{displaymath}
Now, since the sum of the second and third coordinate is equal to $1/2$ we get that the image is equal to the square with vertices in the points 
$(0,0,1/2)$, $(0,1/2,0)$, $(1/2,1/2,0), (1/2,0,1/2)$. The other parts of the contour are handeled analogously.
\end{proof}

\begin{figure}[h]
\begin{center}
\includegraphics[height=4cm]{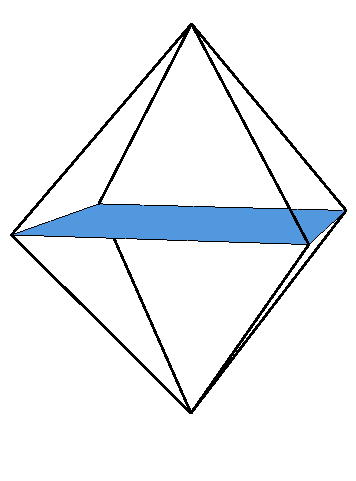}
\includegraphics[height=4cm]{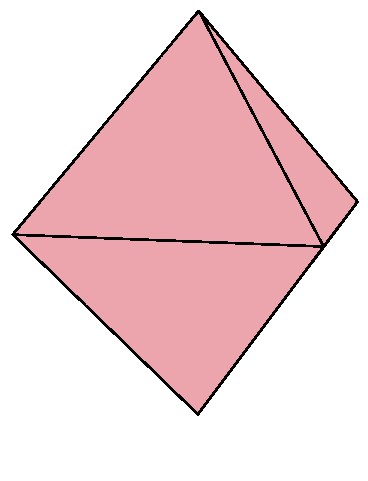}
\includegraphics[height=4cm]{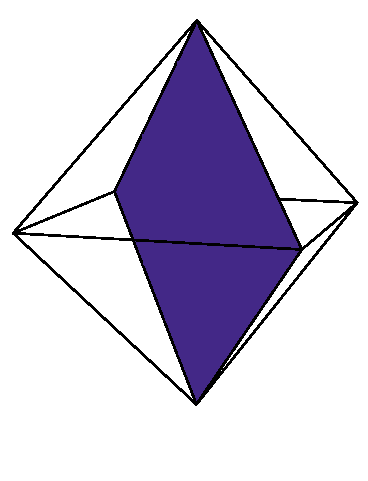}
\caption{The contour minus the boundary of $\bar{\Af}$ when \newline 
$f=1+z+w+t$ is the union of three squares.}
\label{Contour3var}
\end{center}
\end{figure}

\noindent
Let $(x,y,u)$ be a point in the compactified amoeba that is not on the contour. Then $(x,y,u)$ satisfies three inequalities, for example 

\begin{equation}\label{exkammare}
1 + e^x > e^y + e^u, \quad1 + e^y > e^x + e^u, \quad1 + e^u > e^x + e^y.
\end{equation}
If the inequality goes in the direction $>$ we associate a $+$ to it and if it goes in the other direction we associate a $-$ to it. In this way we get a triple with minus or plus signs for every point in the amoeba and thus a numbering of the eight chambers. 
For example, a point $x$ satisfy (\ref{exkammare}) if and only if $x$ belongs to the chamber $(+,+,+)$.


\section{The Ronkin measure in the case of a hyperplane in three variables}\label{Ronkin}

The Ronkin measure for polynomials in two variables is rather well understood. In particular the measure of an affine linear polynomial in two variables is identically equal to $1/\pi^2$ times the Lebesgue measure on the amoeba. Not much is known in the case of three variable polynomials.
A first step is to look at the case where $f$ is a linear polynomial, i.e.,
$f=a+bz+cw+dt$, where $a,b,c$ and $d$ are complex numbers. 
Now, because of Lemma \ref{lattlemma} we only need to consider the case where $a,b,c$ and $d$ all equal 1.


\subsection{The derivatives}
Let $f=1+z+w+t$.
We see that 
\begin{eqnarray*}\label{deriv}
\frac{\partial N_f(x,y,u)}{\partial x} &=& \ddx\left( \frac{1}{2\pi i} \right)^3 \int_{\Log^{-1}(x,y,u)}\log|1+z+w+t|\frac{dzdwdt}{wt}\\
&=&\left( \frac{1}{2\pi i} \right)^3 \int_{\Log^{-1}(x,y,u)} \frac{1}{1+z+w+t}\frac{dzdwdt}{wt}\\
&=&\left( \frac{1}{2\pi i} \right)^2\left(\frac{1}{2\pi i}\right) \int_{\Log^{-1}(x,y,u)} \frac{dz}{z-(-1-w-t)}\frac{dwdt}{wt}.
\end{eqnarray*}

\noindent
Since
the inner integral is equal to $1$ when $|1+e^{y+i\varphi}+e^{u+i\theta}|<e^x$ and is $0$ otherwise we see that
$(\partial/\partial x) N_f$ is equal to the area of the set 
\begin{equation*}
T = \{ (\varphi,\theta)\in\mathbb{T}^2;|1+e^{y+i\varphi}+e^{u+i\theta}|<e^x \},
\end{equation*}
divided by $(2\pi)^2$.
\noindent
Note that $T$ is equal to the area enclosed by the curve that is the projection of the fiber over the point $(x,y,u)$ onto the $\varphi\theta$ plane.

\begin{proposition}\label{derivata}
Outside the contour we have
\begin{align}
&\pi^2\frac{\partial N_f(x,y,u)}{\partial x} =\nonumber \\ 
&= -\int^{r_1}_{r_0}\arccos\left(\frac{1+r^2-e^{2x}}{2r}\right)\frac{d}{dr}\arccos\left(\frac{r^2-e^{2y}-e^{2u}}{2e^{y+u}}\right) dr \label{eq1},
\end{align}
where $r_0$ and $r_1$ depend on which chamber the point $(x,y,u)$ belongs to according to the following table:
\begin{center}
\begin{tabular}{|c|c|c|}
\hline
Chamber &  $r_0$ &  $r_1$ \\
\hline
$(+,+,+)$ & $1-e^x$ & $e^y+e^u$ \\
\hline
$(-,+,+)$ & $1-e^x$ & $1+e^x$ \\
\hline
$(-,-,+)$ & $e^u-e^y$ & $1+e^x$ \\
\hline
$(+,-,+)$ & $e^u-e^y$ & $e^y+e^u$ \\
\hline
$(+,-,-)$ & $e^x-1$ & $e^y+e^u$ \\
\hline
$(+,+,-)$ & $e^y-e^u$ & $e^y+e^u$ \\
\hline
$(-,+,-)$ & $e^y-e^u$ & $1+e^x$ \\
\hline
$(-,-,-)$ & $e^x-1$ & $1+e^x$ \\
\hline
\end{tabular}
\end{center}
The chambers are defined at the end of Section \ref{hyper}.
\end{proposition}

\begin{proof}
We need to calculate the area of $T$ and divide by the area of $\mathbb{T}^2$.
Let $L_{\gamma}$ be the line in the torus defined by $\{  \gamma=\varphi-\theta ; -\pi<\varphi,\theta <\pi  \}$. 
Consider the function $\operatorname{Arm}_{\varphi,\theta}:\mathbb{T}^2\to\C$ given by
\begin{equation*}
\operatorname{Arm}_{y,u}(\varphi,\theta) = 1+e^{y+i\varphi}+e^{u+i\theta}.
\end{equation*}
A straight forward calculation gives that the Jacobian of that function
is constant along $L_{\gamma}$. If $D(a,b)$ is the disc with center $a$ and radii $b$, then this means that
\begin{eqnarray*}
\frac{\Length(L_{\gamma}\cap T)}{\Length(L_{\gamma})}&=&
\frac{\Length(\operatorname{Arm}_{\varphi,\theta}(L_{\gamma})\cap D(0,e^x))}{\Length(\operatorname{Arm}_{\varphi,\theta}(L_{\gamma}))}  \\
&=& \frac{\Length(\partial D(1,r)\cap D(0,e^x))}{\Length(\partial D(1,r))}=\frac{\alpha}{\pi},
\end{eqnarray*}
where $\alpha$ is the angle that $w+t$ must have precisely to hit $D(0,e^x)$ and where $r=|w+t|$. Integrating $\alpha$ over $\gamma$ when $0\leq \gamma \leq \pi$ we get 
\begin{equation*}
\frac{\partial N_f}{\partial x} = \frac{1}{2\pi^2}\int_0^{2\pi}\alpha(\gamma)d\gamma = \frac{1}{\pi^2}\int_0^{\pi}\alpha(\gamma)d\gamma
\end{equation*}
for symmetry reasons.
Now, rewrite $\alpha$ and $\gamma$ in terms of $r$ just by solving the triangles in Figure \ref{fig:BevisTriangel}.

\begin{figure}[h]
\begin{center}
\includegraphics[height=5cm]{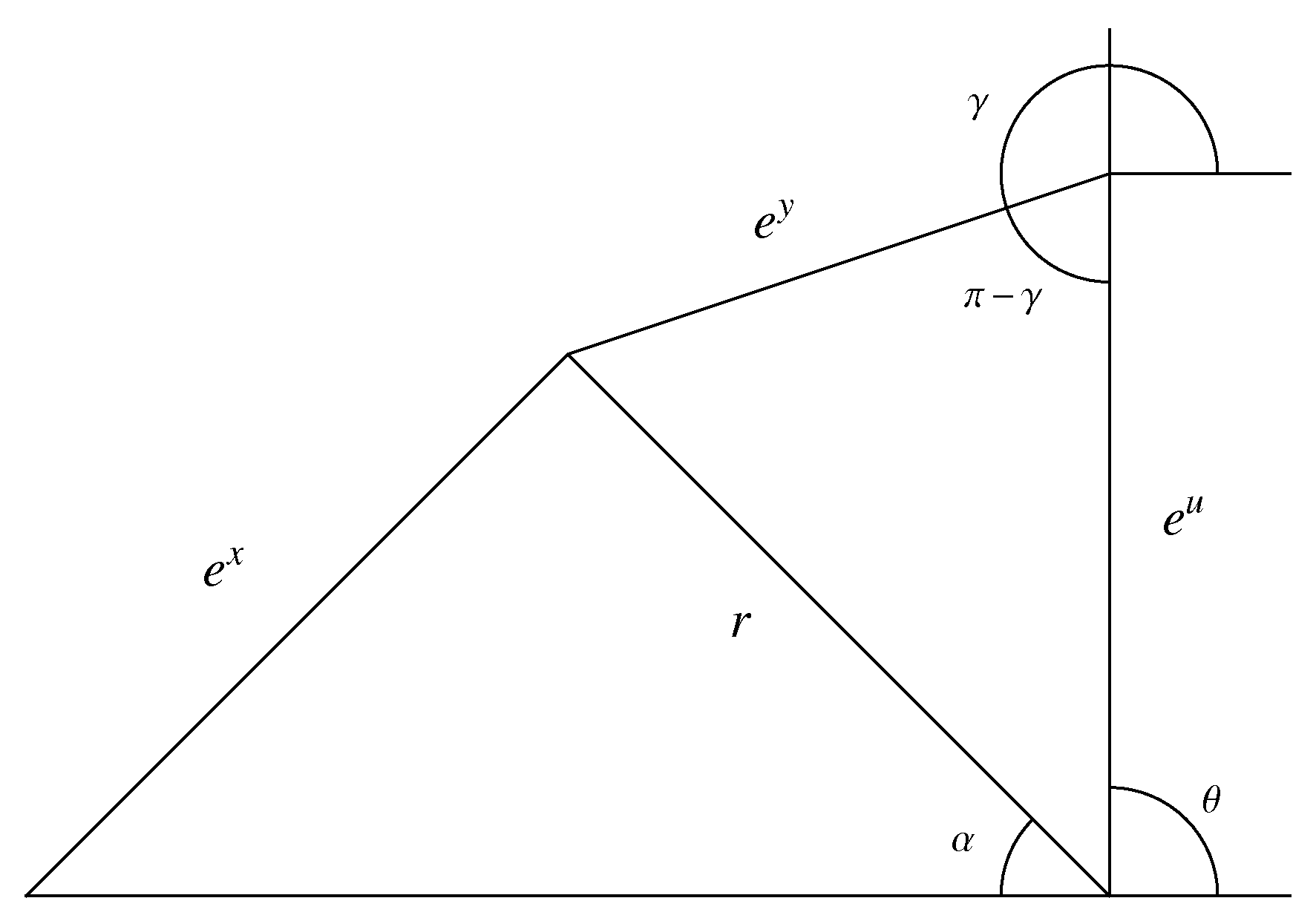}
\caption{}
\label{fig:BevisTriangel}
\end{center}
\end{figure}

This gives
\begin{eqnarray*}
&&\alpha = \arccos\left(\frac{1+r^2-e^{2x}}{2r}\right) \quad \text{ and} \\
&&\gamma = \arccos\left(\frac{r^2-e^{2y}-e^{2u}}{2e^{y+u}}\right).
\end{eqnarray*}
The only thing left to do is to figure out what the integration limits should be.
Let $\mathcal{W}$ be the image of the function $\operatorname{Arm}_{y,u}$. Figures \ref{1} - \ref{6} represent $\mathcal{W}$ and $\partial D(0,e^x)$ in the different chambers, and since the integration is over $r$ corresponding to points on the intersection of $\mathcal{W}$ and $\partial D(0,e^x)$ the integration limits can easily be seen in the figures.
Note that the minus sign comes from the fact that the integration limits should change places to get the ones in the theorem.
\end{proof}

\begin{figure}[h]
\begin{minipage}[t]{.325\linewidth}
\centering \includegraphics[width=3cm]{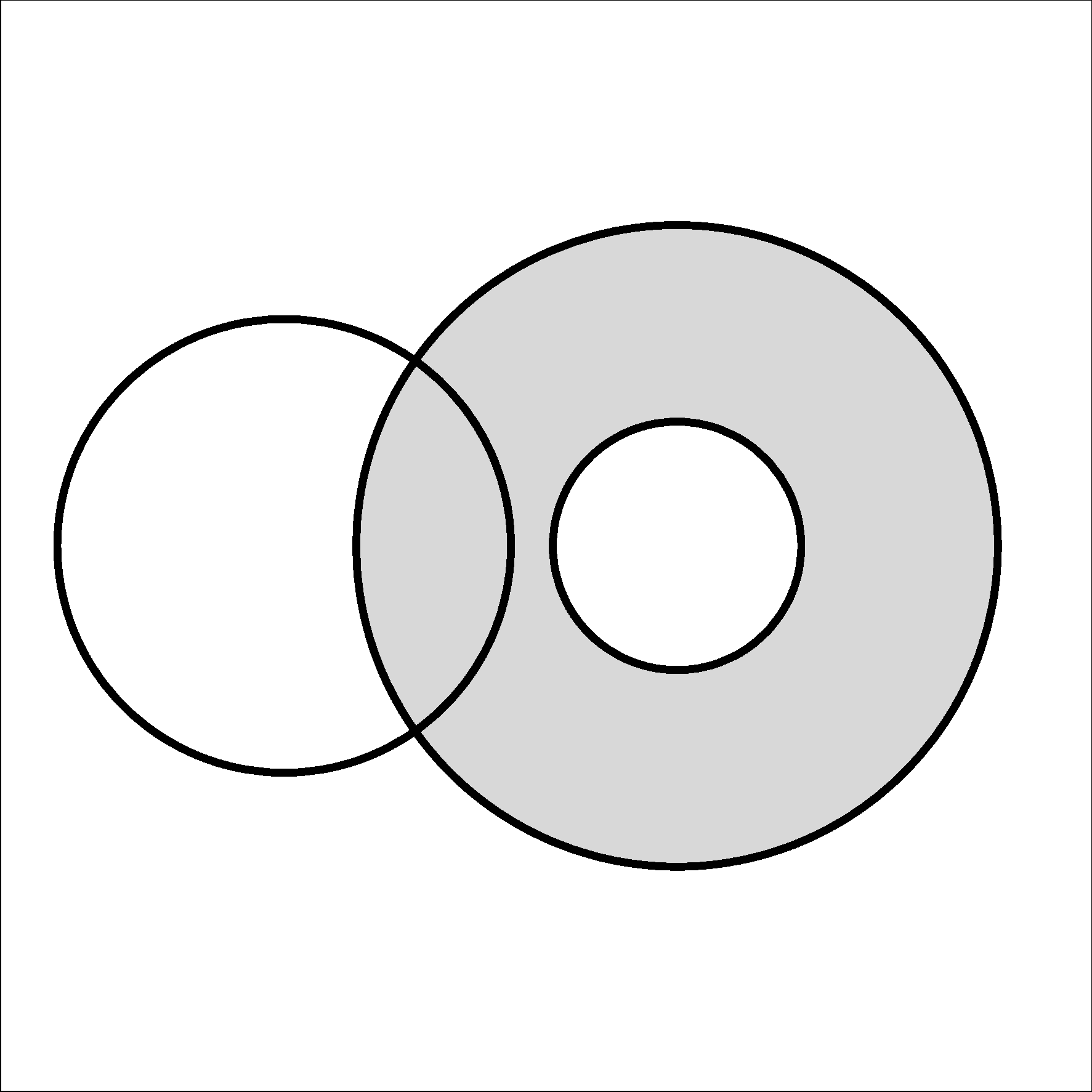}
\caption{$(+,+,+)$}\label{1} 
\end{minipage}
\begin{minipage}[t]{.325\linewidth}
\centering \includegraphics[width=3cm]{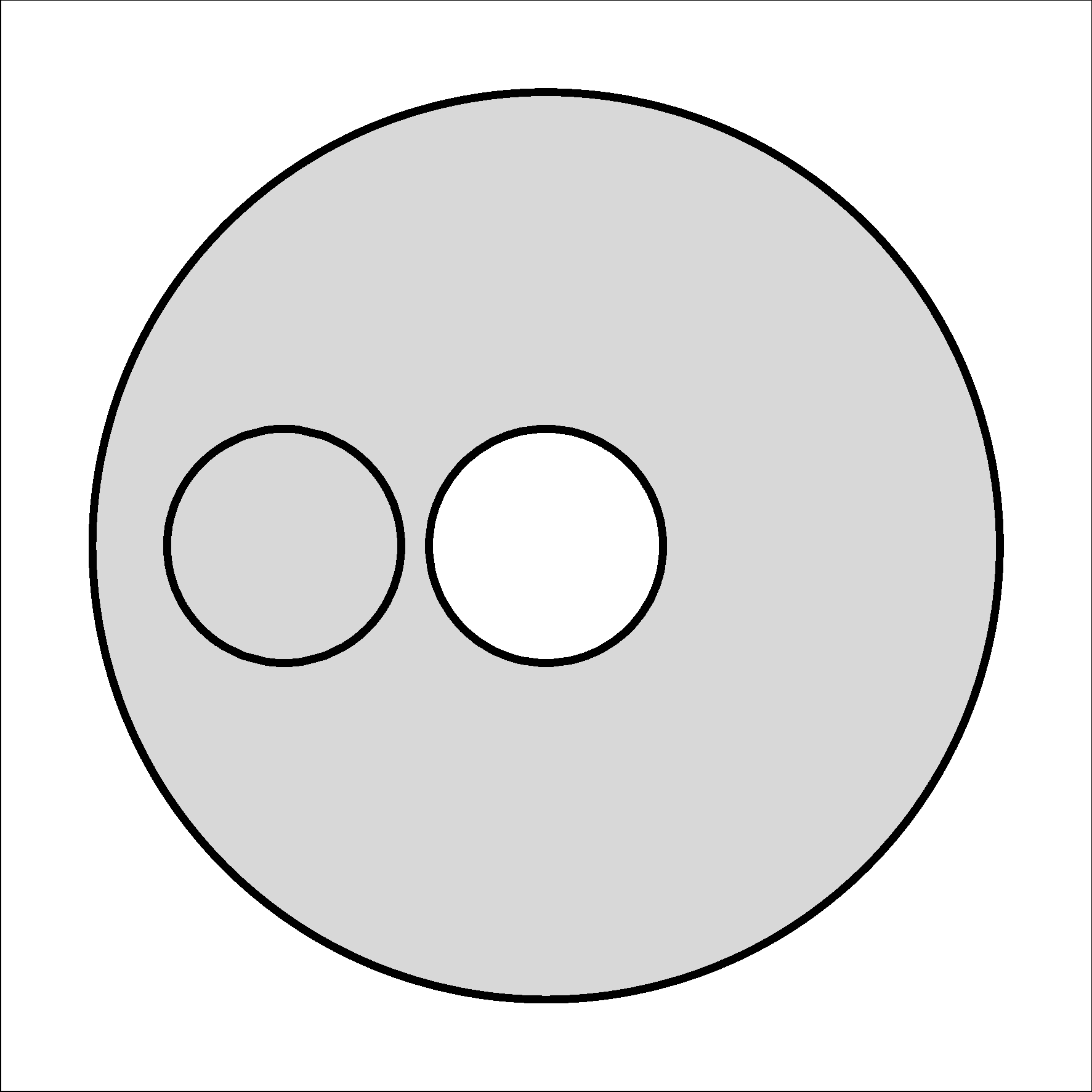}
\caption{$(-,+,+)$} \label{2}
\end{minipage}
\begin{minipage}[t]{.325\linewidth}
\center{ \includegraphics[width=3cm]{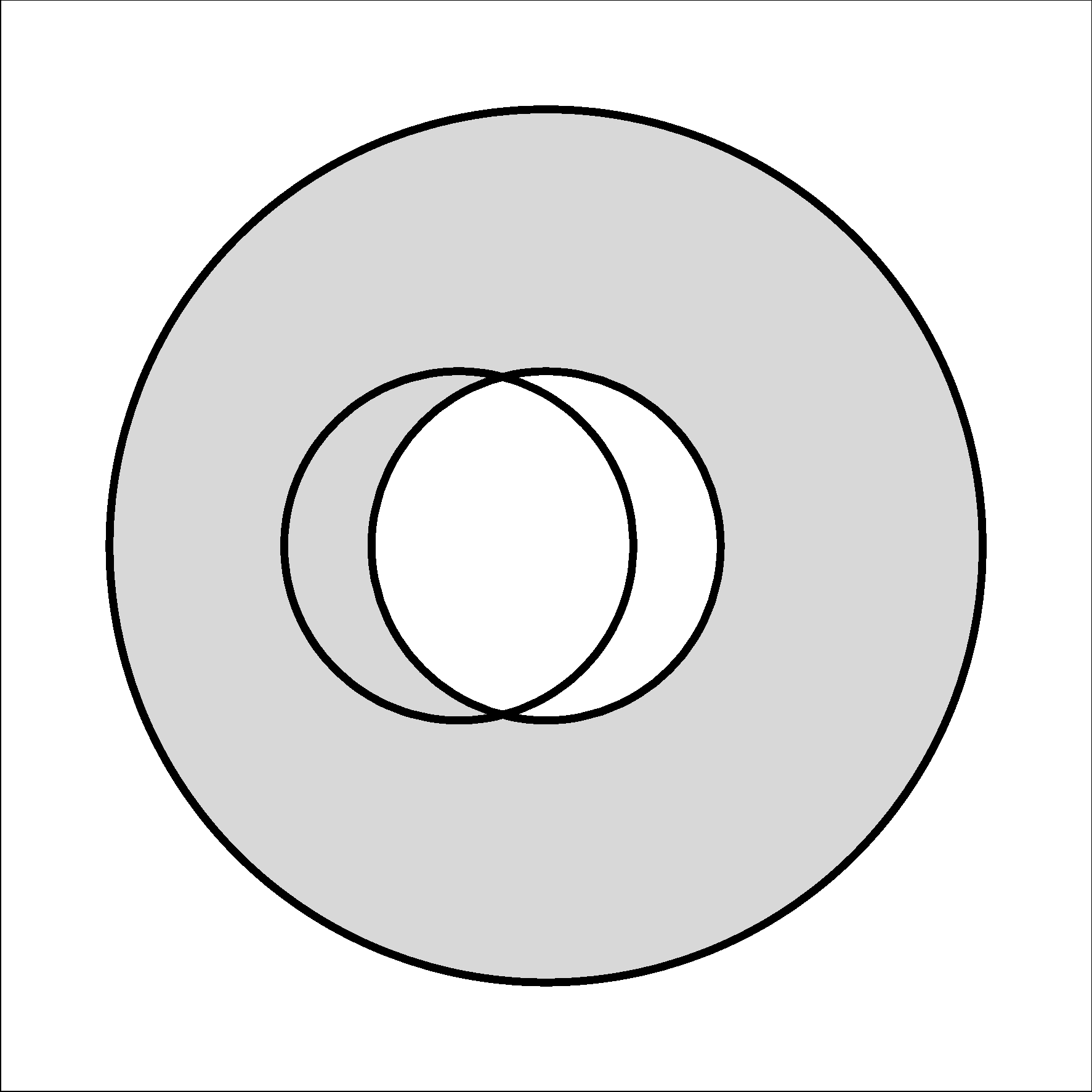} 
\caption{$(-,-,+), (-,+,-)$}\label{3}}
\end{minipage}
\end{figure}

\begin{figure}[h]
\begin{minipage}[t]{.323\linewidth}
\centering \includegraphics[width=3cm]{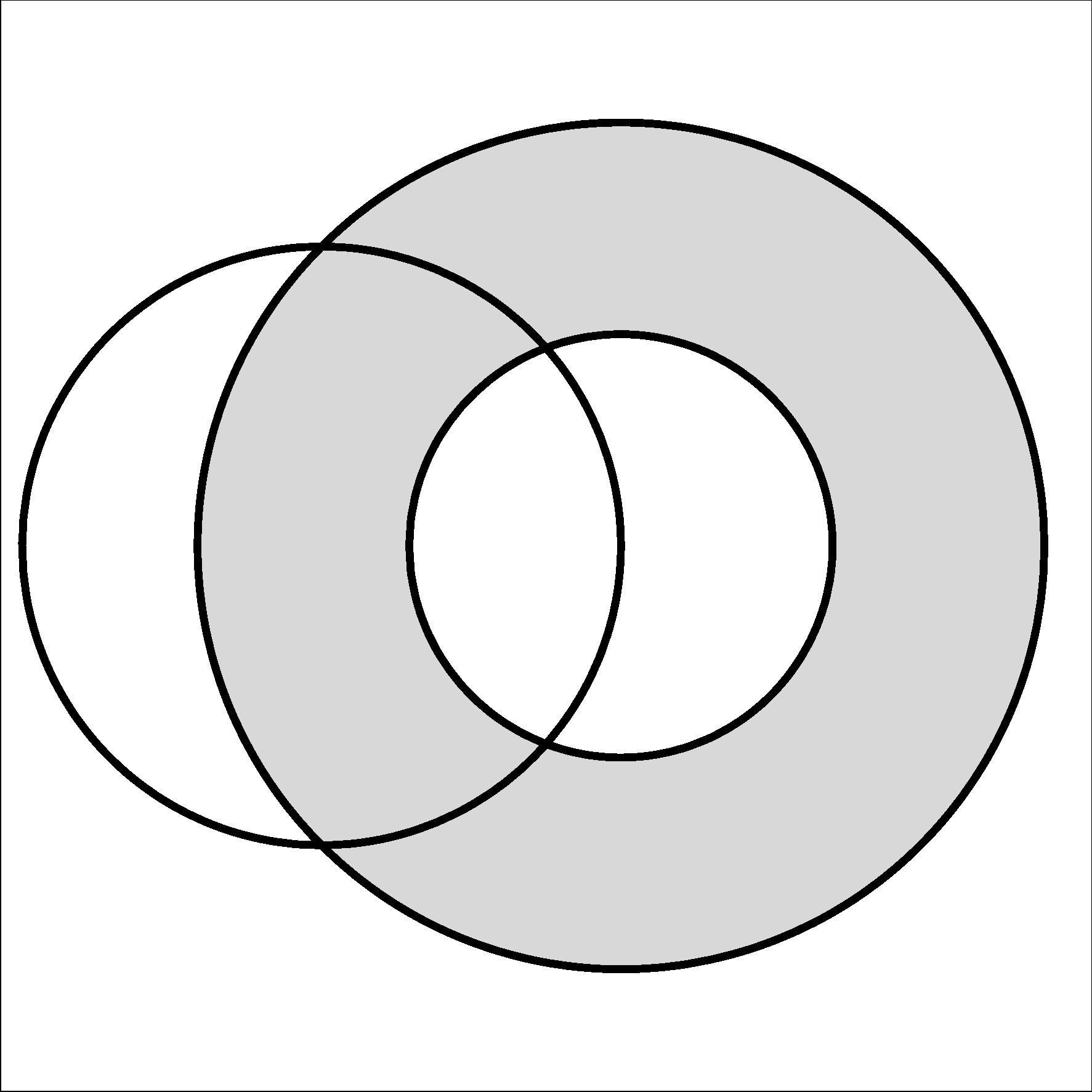} 
\caption{$(+,-,+),(+,+,-)$}\label{4}
\end{minipage}
\begin{minipage}[t]{.323\linewidth}
\centering \includegraphics[width=3cm]{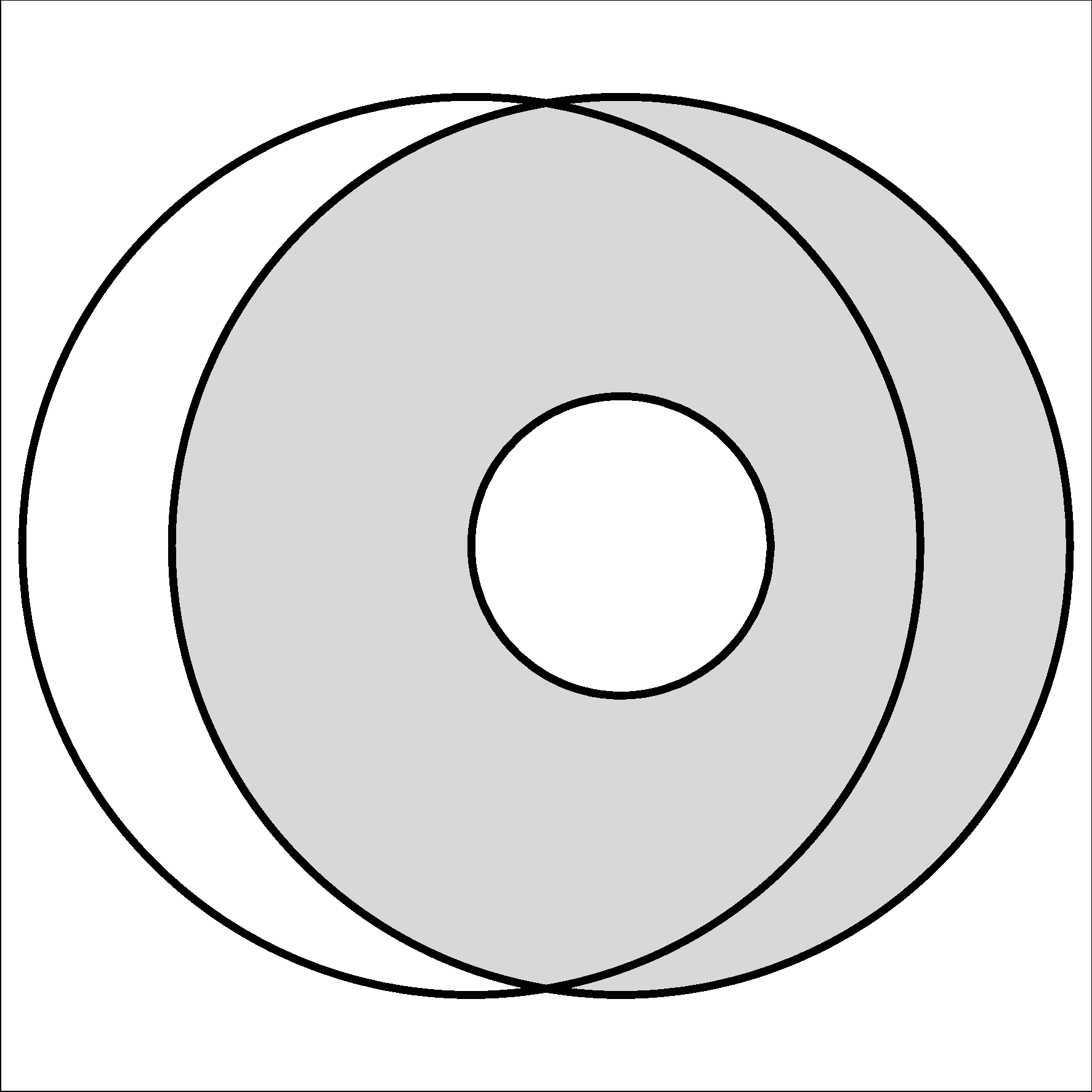} 
\caption{$(+,-,-)$}\label{5}
\end{minipage}
\begin{minipage}[t]{.323\linewidth}
\centering \includegraphics[width=3cm]{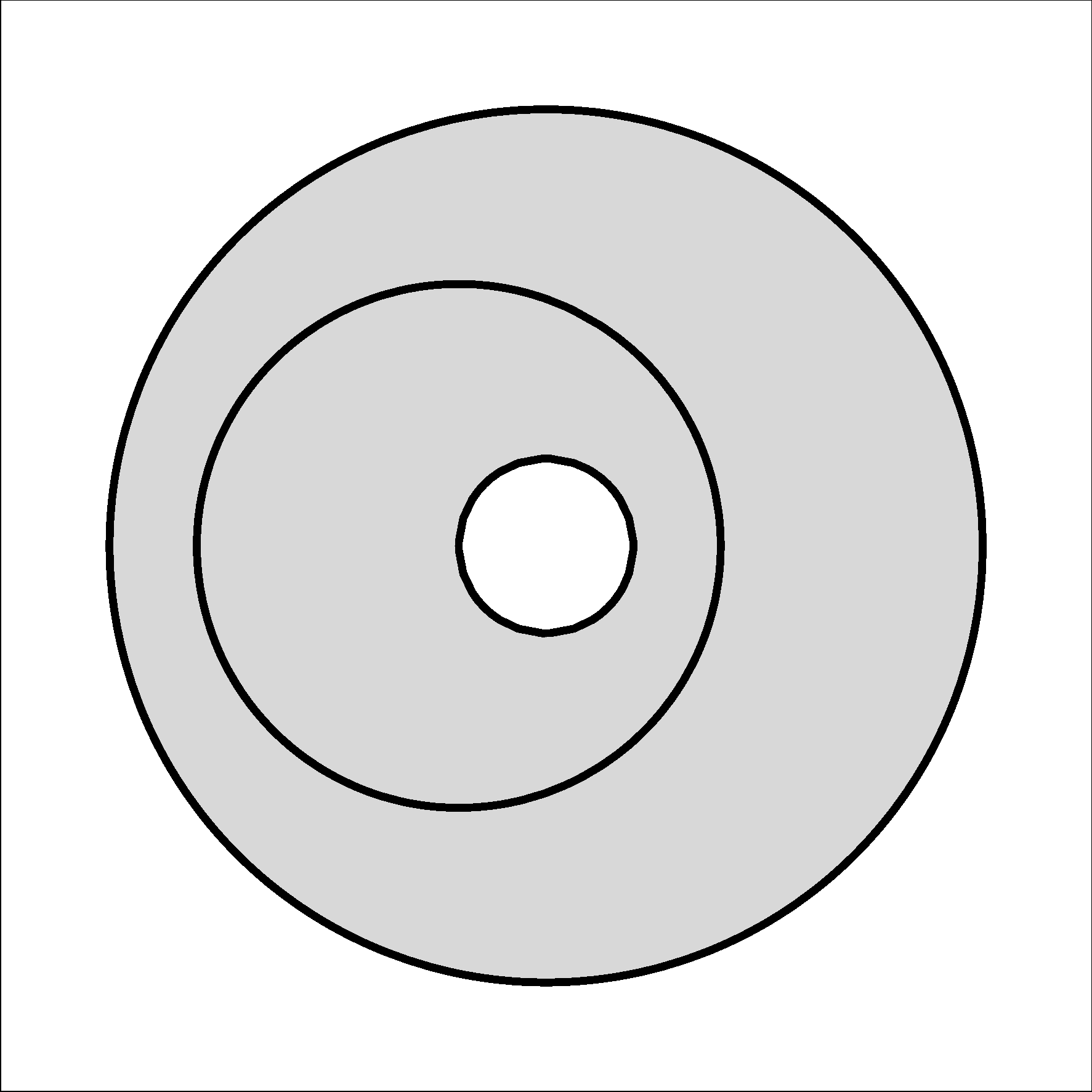} 
\caption{$(-,-,-)$}\label{6}
\end{minipage}
\end{figure}

\begin{remark}
{\rm
It should be remarked that the integral (\ref{eq1}), and hence the derivatives, are continuous inside the amoeba, even at the contour. This can be seen from the fact that 
the integral in (\ref{eq1}) can be written on the form 
\begin{equation*}
\int_{r_0^2}^{r_1^2}\frac{-\alpha(x,y,u,\sqrt{s})}{\sqrt{-(s-(e^y+e^u)^2)(s-(e^y-e^u)^2)}}ds,
\end{equation*}
see the calculations in Lemma \ref{teklemma2} below,
and that this integral is bounded in the closure of each chamber.
}
\end{remark}

Let
\begin{align}
&\phi(r,x,y,u) := \arccos\left(\frac{1+r^2-e^{2x}}{2r}\right) \quad \text{and } \label{phi}\\ 
&\psi(r,x,y,u) := \arccos\left(\frac{r^2-e^{2y}-e^{2u}}{2e^{y+u}}\right).\label{psi}
\end{align}
\noindent
Then even though $x$ and $y$ appear in the integration limits $r_0$ and $r_1$ we get the following lemma.

\begin{lemma}\label{teklemma}
Outside of the contour, and for $r_0$ and $r_1$ as above,
\begin{eqnarray*}
&&\frac{\partial }{\partial x}\int_{r_0}^{r_1}\phi\frac{d}{dr}\psi dr = \int_{r_0}^{r_1}\frac{\partial }{\partial x} \phi\frac{d}{dr}\psi dr\quad \text{and} \\
&&\frac{\partial }{\partial y}\int_{r_0}^{r_1}\phi\frac{d}{dr}\psi dr = -\int_{r_0}^{r_1}\frac{\partial }{\partial y} \psi\frac{d}{dr}\phi dr,
\end{eqnarray*}
where $\phi$ and $\psi$ are defined by (\ref{phi}) and (\ref{psi}).
\end{lemma}

\begin{proof}
We prove the second equality. The first is proved along the same lines.
The lemma follows if we prove it for the case when both $r_0$ and $r_1$ depend on $y$, i.e.,
\begin{equation*}
r_1 = e^x+e^y, \qquad r_0 = \pm e^y-e^u.
\end{equation*}
We first note that 
\begin{align}
e^y\ddy \phi(r_1) &= \left( \ddr\phi\right) (r_1) \quad \text{and} \label{lem1} \\ 
-e^y\ddy \phi(r_0) &= \pm \left( \ddr\phi \right)(r_0) \label{lem2}.
\end{align} 
We want to prove that 
\begin{equation}\label{lemekv}
\ddy \int_{r_0}^{r_1}\phi\ddr \psi dr + \int_{r_0}^{r_1}\ddy \psi \ddr \phi dr = 0.
\end{equation}

By using integration by parts and the definition of derivatives the left hand side of \eqref{lemekv} is
\begin{align*}
& \frac{\partial}{\partial y}\left( \left[ \phi\psi \right]_{r_0}^{r_1} - \int_{r_0}^{r_1} \psi\frac{d}{dr}\phi dr \right) + 
\int_{r_0}^{r_1}\frac{\partial}{\partial y} \psi \frac{d}{dr}\phi dr \\
= & \frac{\partial}{\partial y} \left[ \psi\phi\right]_{r_0}^{r_1}  \\
- & \lim_{h\to 0}\left( \frac{1}{h}\int_{r_0(y+h)}^{r_1(y+h)} \psi(y+h)\ddr \phi dr 
-\frac{1}{h}\int_{r_0(y)}^{r_1(y)} \psi(y)\ddr \phi dr\right)\\
+ & \int_{r_0(y)}^{r_1(y)} \lim_{h\to 0} \frac{\psi(y+h)-\psi(y)}{h}\ddr \phi dr.
\end{align*}
By linearity this is equal to

\begin{align*}
&\frac{\partial}{\partial y} \left[ \phi\psi \right]_{r_0}^{r_1}  + \lim_{h \to 0}\frac{1}{h}\int_{r_0(y)}^{r_0(y+h)}\psi(y+h)\ddr\phi dr  \\
-&\lim_{h\to 0}\frac{1}{h}\int_{r_1(y)}^{r_1(y+h)}\psi(y+h)\ddr \phi dr,
\end{align*}
and since $\psi$ is bounded we get that the left hand side of \eqref{lemekv} is equal to
\begin{align*}
&\ddy \left[ \phi\psi \right]_{r_0}^{r_1} + \lim_{h\to 0}\frac{1}{h}\left( r_0(y+h)-r_0(y) \right) \psi\left( \ddr\phi \right)\bigg|_{r_0} \\
-&\lim_{h\to 0}\frac{1}{h}\left( r_1(y+h)-r_1(y) \right) \psi\left( \ddr\phi \right) \bigg|_{r_1}.
\end{align*}
Now,
\begin{eqnarray*}
&&\frac{1}{h}(r_1(y+h)-r_1(y)) = \frac{1}{h}(e^y(e^h-1))\to e^y \quad \text{when}\quad h\to 0 \quad \text{and}\\
&&\frac{1}{h}(r_0(y+h)-r_0(y)) = \pm \frac{1}{h}(e^y(e^h-1))\to \pm e^y \quad \text{when}\quad h\to 0,
\end{eqnarray*}
so \eqref{lem1} and \eqref{lem2} shows that \eqref{lemekv} holds.
\end{proof}

Lemma \ref{teklemma} will be useful to calculate the second order derivatives of $N_f$.

\begin{lemma}\label{teklemma2}
For $(x,y,u)\in\Af\setminus\mathcal{C}$ and with $r_0$ and $r_1$ as above we have
\begin{align}
&\frac{\partial^2 N_f}{\partial x^2} = \frac{2e^{2x}}{\pi^2}\int_{r_0^2}^{r_1^2}\frac{1}{\sqrt{(s-A)(s-B)(s-C)(s-D))}}ds \quad\text{and} \label{integralett} \\
&\frac{\partial^2 N_f}{\partial x\partial y} = \frac{-1}{2\pi^2}\int_{r_0^2}^{r_1^2}\frac{s^2 +P_1 s + P_2}{s\sqrt{(s-A)(s-B)(s-C)(s-D))}}ds, \label{integraltva}
\end{align}
where
\begin{align*}
A & = (1+e^x)^2, \quad B = (e^y+e^u)^2, \\
C & = (1-e^x)^2,\quad D = (e^y-e^u)^2
\end{align*}
and 
\begin{equation*}
P_1 = (e^{2x}+e^{2y}-1-e^{2u}),\quad P_2 = (1+e^x)(1-e^x)(e^y+e^u)(e^u-e^y).
\end{equation*}
\end{lemma}

\begin{proof}
We start with the first equality. By Proposition \ref{derivata} and Lemma~\ref{teklemma} we get that $(\partial^2 /\partial x^2) N_f$ is equal to
\begin{equation*} -\frac{1}{\pi^2}\int_{r_0}^{r_1}\ddx\arccos\left(\frac{1+r^2-e^{2x}}{2r}\right)\ddr\arccos\left(\frac{r^2-e^{2y}-e^{2u}}{2e^{y+u}}\right) dr.
\end{equation*}
An easy calculation shows that
\begin{eqnarray*}
\ddx\arccos\left(\frac{1+r^2-e^{2x}}{2r}\right) &=& \frac{2e^{2x}}{\sqrt{4r^2 -(1+r^2-e^{2x})^2}},\\
\ddr\arccos\left(\frac{r^2-e^{2y}-e^{2u}}{2e^{y+u}}\right) &=& \frac{-2r}{\sqrt{4e^{2(y+u)}-(r^2-e^{2y}-e^{2u})^2}}.
\end{eqnarray*}
Now, make the change of variables $s=r^2$ and make use of the formula 
\begin{eqnarray*}
4a^2b^2 - (c^2-a^2-b^2)^2 &=& - (a^2-(b+c)^2)(a^2-(b-c)^2) \\
&=& - (b^2-(a+c)^2)(b^2-(a-c)^2)\\
&=& - (c^2-(a+b)^2)(c^2-(a-b)^2)
\end{eqnarray*}
that is valid for all $a$ and $b$. The first equation in the lemma is thereby proved. 
The second equation is proved in a similar way.
\end{proof}
\noindent
Note that $r_1^2$ is either $A$ or $B$ and $r_0^2$ is either $C$ or $D$.
We see that the integrals in (\ref{integralett}) and (\ref{integraltva}) depend on $x,y$ and $u$ in a smooth manner except at the singular points where $A=B, C=D, B=C$  and possibly 
when $r_0 = 0$, i.e., when $1=e^x$ or when $e^y=e^u$.
But $P_2=0$ at the points where $1=e^x$ or when $e^y=e^u$, and thus there might be that the integral converges anyway. That is actually the case.
To see this it is enough to realize that 
\begin{equation*}
\lim_{\epsilon \to 0}\int_{\epsilon}^M\frac{\epsilon}{s\sqrt{s-\epsilon}}ds = 0,
\end{equation*}
for some constant $M\neq 0$. But that is true because
\begin{align*}
\lim_{\epsilon \to 0}\int_{\epsilon}^M\frac{\epsilon}{s\sqrt{s-\epsilon}}ds=\lim_{\epsilon \to 0}\sqrt{\epsilon}\int_{1}^{M/\epsilon}\frac{1}{s\sqrt{s-1}}ds.
\end{align*}
\noindent
Now, a similar argument gives that $\partial^2 N_f / \partial x\partial y$ not only is continuous but also smooth at the points where $1=e^x$ and $e^y=e^u$. Note that the equality $B=C$ hold exactly on the boundary of the amoeba and that the equations  $A=B$ and $C=D$ are true exactly on the other part of the contour.
We therefore have the following proposition.

\begin{proposition}
Let $f=1+z+w+t$. Then $\mu_f$ is smooth outside the contour of the amoeba of $f$.
\end{proposition}


\subsection{Connections to elliptic integrals}

Elliptic integrals naturally comes up in many situations. For example when calculating the arc length of an ellipse (hence the name).
Lemma \ref{teklemma2} says that the second order derivatives of the Ronkin function of an affine linear polynomial in three variables are complete elliptic integrals.

\begin{definition}\label{elliptic}
An \rm elliptic integral \it is an integral of the form \\$\int R(s,\sqrt{P(s)})$ where $P$ is a polynomial of degree $3$ or $4$ with no multiple roots and $R$ is a rational function of $s$ and $\sqrt{P}$. It is always possible to express elliptic integrals as linear combinations in terms of elementary functions and the following three integrals:
\begin{eqnarray*}
\K(\varphi,k) & := & \int_0^{\varphi}\frac{d\theta}{\sqrt{1-k^2\sin^2\theta}}=\int_0^t\frac{ds}{\sqrt{(1-s^2)(1-k^2s^2)}},\\
\operatorname{E}(\varphi,k) & := & \int_0^{\varphi}\sqrt{1-k^2\sin^2\theta}d\theta = \int_0^t\sqrt{\frac{1-k^2s^2}{1-s^2}}ds,\\
\Pi(\varphi,\alpha^2,k) & := & \int_0^{\varphi}\frac{d\theta}{(1-\alpha^2\sin^2\theta)\sqrt{1-k^2\sin^2\theta}}=\\
& = & \int_0^t\frac{ds}{(1-\alpha^2s^2)\sqrt{(1-s^2)(1-k^2s^2)}}.
\end{eqnarray*}
The integrals above are said to be on normal, or Legendre, form.
If $\varphi=\frac{\pi}{2}$ we say that the integrals are \rm complete \it and we denote the three complete integrals on normal form by $\K(k)$, $\operatorname{E}(k)$ and $\Pi(\alpha^2,k)$, respectively.
\end{definition}

\begin{lemma}\label{teklemma3}
Assume $a>b>c>d$. Then 
\[\int_c^b \frac{s^jds}{\sqrt{(s-a)(s-b)(s-c)(s-d)}},\quad j=-1,0,1,\]
transforms into the following complete elliptic integrals on normal form:
\begin{eqnarray*}
&&g\K(k) \quad\qquad\qquad\qquad\quad\quad\quad\quad \textrm{  if } j=0 \\
&&dg\K(k)+g(c-d)\Pi(\alpha^2,k) \quad\quad\quad \!\textrm{if } j=1\\
&&\frac{g}{d}\K(k)+g(\frac{1}{c}-\frac{1}{d})\Pi(\alpha^2\frac{d}{c},k)  \quad\quad \textrm{if } j=-1,
\end{eqnarray*}
where 
\begin{displaymath}
k^2=\frac{(b-c)(a-d)}{(a-c)(b-d)},\qquad\alpha^2=\frac{b-c}{b-d},\qquad g=\frac{2}{\sqrt{(a-c)(b-d)}}.
\end{displaymath}
\end{lemma}
\noindent
These results are well-known, see for example \cite{Byrd}.

Lemmas \ref{teklemma2} and \ref{teklemma3} make it possible to express the second order derivatives of $N_f$ in terms of complete elliptic integrals of the first and third kind. The only thing one has to do is to determine how $A,B,C,D$ in Lemma \ref{teklemma2} are ordered. In chamber $(+,+,+)$ we see that $A>B>C>D$ for example. 
Determining the order of $A,B,C$ and $D$ gives us the following expressions of the second order derivatives in the different chambers.
\begin{proposition}\label{andraderivator}
Let $f=1+z+w+t$. The second order derivatives of the Ronkin function $N_f$ can be expressed in terms of complete elliptic integrals of the first and third kind, as
\begin{align*}
\frac{\partial^2 N_f}{\partial x^2} &= \frac{2ge^{2x}}{\pi^2}\K(k), \\
\frac{\partial^2 N_f}{\partial x \partial y} &= \frac{-g}{2\pi^2}\left( Q_1\K(k) + Q_2\Pi(\alpha_1^2,k) +Q_3\Pi(\alpha_2^2,k)\right),
\end{align*}
where $k^2,\alpha_1^2\alpha_2^2,g^2,Q_1,Q_2$ and $Q_3$ are rational functions in $e^x,e^y$ and $e^u$ depend on what chamber $(x,y,u)$ lies in.
With
\begin{equation*}
\xi:=(1+e^x+e^y-e^u)(1+e^x-e^y+e^u)(1-e^x+e^y+e^u)(-1+e^x+e^y+e^u),
\end{equation*}
these functions will take the form according to the following:
\newline
In the chambers $(+,+,+)$ and $(+,-,-)$,
\begin{align*}
g&=\frac{1}{2\sqrt{e^{x+y+u}}},\quad
k^2=\frac{\xi}{16e^{x+y+u}},\\
Q_1&=2 \frac{e^y(e^{2x}+e^{2y}+e^{2u}-1-2e^{y+u})}{(e^y-e^u)},\quad
Q_2= (1-e^x+e^y-e^u)(1-e^x-e^y+e^u),\\
Q_3&=\frac{(e^u+e^y)(1-e^x+e^y-e^u)(1+e^x)(1-e^x-e^y+e^u)}{(e^u-e^y)(e^x-1)},\\
\alpha_1^2&=\frac{(1-e^x+e^y+e^u)(-1+e^x+e^y+e^u)}{4e^{y+u}},\quad
\alpha_2^2=\alpha_1^2\frac{(e^y-e^u)^2}{(1-e^x)^2}.
\end{align*}
In the chambers $(-,+,+)$ and $(-,-,-)$,
\begin{align*}
g&=\frac{2}{\sqrt{\xi}},\quad
k^2=\frac{16e^{x+y+u}}{\xi},\\
Q_1&=2 \frac{e^y(e^{2x}+e^{2y}+e^{2u}-1-2e^{y+u})}{(e^y-e^u)},\quad
Q_2= (1-e^x+e^y-e^u)(1-e^x-e^y+e^u),\\
Q_3&=\frac{(e^u+e^y)(1-e^x+e^y-e^u)(1+e^x)(1-e^x-e^y+e^u)}{(e^u-e^y)(e^x-1)},\\
\alpha_1^2&=\frac{4e^x}{(1+e^x+e^y-e^u)(1+e^x-e^y+e^u)},\quad
\alpha_2^2=\alpha_1^2\frac{(e^y-e^u)^2}{(1-e^x)^2}.
\end{align*}
In the chambers $(-,-,+)$ and $(-,+,-)$,
\begin{align*}
g&=\frac{1}{2\sqrt{e^{x+y+u}}},\quad
k^2=\frac{\xi}{16e^{x+y+u}},\\
Q_1&=2 \frac{e^x(e^{2x}+e^{2y}-e^{2u}+1-2e^{x})}{(e^x-1)},\quad
Q_2= -(1-e^x+e^y-e^u)(1-e^x-e^y+e^u),\\
Q_3&=-\frac{(e^u+e^y)(1-e^x+e^y-e^u)(1+e^x)(1-e^x-e^y+e^u)}{(e^u-e^y)(e^x-1)},\\
\alpha_1^2&=\frac{(1+e^x-e^y+e^u)(1+e^x+e^y-e^u)}{4e^{x}},\quad
\alpha_2^2=\alpha_1^2\frac{(e^x-1)^2}{(e^y-e^u)^2}.
\end{align*}
In the chambers $(+,-,+)$ and $(+,+,-)$,
\begin{align*}
g&=\frac{2}{\sqrt{\xi}},\quad
k^2=\frac{16e^{x+y+u}}{\xi},\\
Q_1&=-2 \frac{e^x(e^{2x}+e^{2y}-e^{2u}+1-2e^{x})}{(1-e^x)},\quad
Q_2= -(1-e^x+e^y-e^u)(1-e^x-e^y+e^u),\\
Q_3&=-\frac{(e^u+e^y)(1-e^x+e^y-e^u)(1+e^x)(1-e^x-e^y+e^u)}{(e^u-e^y)(e^x-1)},\\
\alpha_1^2&=\frac{4e^{y+u}}{(1-e^x+e^y+e^u)(-1+e^x+e^y+e^u)},\quad
\alpha_2^2=\alpha_1^2\frac{(1-e^x)^2}{(e^y-e^u)^2}.
\end{align*}
\end{proposition}
\noindent
Even though it appears that the mixed second order derivative of $N_f$ is singular at the points $(x,y,u)\in\R^3$ where $e^x=1$ or $e^y=e^u$ we saw that $P_2$ in Lemma \ref{teklemma2} vanishes at those points. This means that $Q_3 = 0$ and that $Q_1$ take the form $g(1+e^x)(1-e^x)$, and thus is not singular.
\newline
\newline
A priori we know that the Hessian matrix will be symmetric in every chamber. This gives us several relations between 
elliptic integrals of the first and third kind that as far as we know cannot be explained by the known relations that can be found in the literature.
There might thus be some interesting relations hidden in the following equation that we get by considering the case of the chamber $(+,+,+)$.

\begin{example}{\rm
For $a,b,c>0$ that satisfy the inequalities 
$1+a>b+c,\quad1+b>a+c,\quad1+c>a+b$, we get that
\begin{eqnarray*}
&&2\frac{(1+a+b-c)(a-b)c}{(a-c)(c-b)}\K(k) + (1-a+b-c)\Pi\left( \alpha_1^2,k\right)\\
&&-(1+a-b-c)\Pi\left( \alpha_2^2,k\right)
+\frac{(1+a)(b+c)(1-a+b-c)}{(1-a)(b-c)}\Pi\left( \alpha_3^2,k\right)\\
&&-\frac{(1+b)(a+c)(1+a-b-c)}{(1-b)(a-c)}\Pi\left( \alpha_4^2,k\right)\equiv0,
\end{eqnarray*}
with
\begin{eqnarray*}
&&k^2 = \frac{(1+a+b-c)(1+a-b+c)(1-a+b+c)(-1+a+b+c)}{16abc}, \\
&&\alpha_1^2 =\frac{(1-a+b+c)(-1+a+b+c)}{4bc},\\
&&\alpha_2^2 =\frac{(1+a-b+c)(-1+a+b+c)}{4ac},\\
&&\alpha_3^2 =\frac{(1-a+b+c)(-1+a+b+c)(b-c)^2}{4bc(1-a)^2},\\
&&\alpha_4^2 =\frac{(1+a-b+c)(-1+a+b+c)(a-c)^2}{4ac(1-b)^2}.
\end{eqnarray*}
}
\end{example}


\subsection{Connections to hypergeometric functions}

Elliptic integrals are special cases of hypergeometric functions, which are 
very important in the field of special functions and mathematical physics.\\

Remember that the Gauss hypergeometric function, $\gauss$, and the Appell hypergeometric function in two variables, $\appellett$, is defined by the series 
\begin{equation}\label{gauss}
\gauss(a,b;c;z) = \sum_{n=0}^{\infty}\frac{(a)_n(b)_n}{(c)_n}\frac{z^n}{n!},
\end{equation}
where $(\lambda)_n = \Gamma(\lambda+n)/\Gamma(\lambda)$ is the Pochhammer symbol, and
\begin{equation*}
\appellett(a,b,b^{\prime};c;z;w)= \sum_{m,n=0}^{\infty}\frac{(a)_{m+n}(b)_m(b^{\prime})_n}{(c)_{m+n}}\frac{z^m w^n}{m!n!}.
\end{equation*}
Here the parameter $c$ is assumed not to be a non-positive integer.
The radius of convergence for $\gauss$ is $1$ unless $a$ or $b$ is a non-positive integer, in which case the radius of convergence is infinite, and the series $\appellett$ 
converge for $|z|<1$ and $|w|<1$,
It is well-known that the elliptic integrals are hypergeometric and that
\begin{equation}\label{gaussAhyp}
\K(k) = \frac{\pi}{2} \gauss(1/2,1/2;1;k^2)
\end{equation}
and
\begin{equation}\label{tredjeochappell}
\Pi(\alpha^2,k) = \frac{\pi}{2}{\rm F}_1(1/2;1,1/2;1;\alpha^2,k^2),
\end{equation}
see for example \cite{Exton}.

Gelfand, Kapranov and Zelevinsky revolutionized the theory of hypergeometric functions by considering a system of differential equations in several variables, see \cite{GKZ2}. The solutions to that specific system, called the GKZ-system, have certain homogeneities and they are defined to be $A$-hypergeometric or GKZ-hypergeometric functions.
By dehomogenizing these functions one can get $\gauss$ and Appell functions and many other generalizations of the Gauss hypergeometric function.\\
Following \cite{Lisa}, given a $(n\times N)$-matrix $A$ on the form
\begin{displaymath}
A =
\left( \begin{array}{cccc}
1 & 1 & \ldots & 1 \\
\alpha^1 & \alpha^2 & \ldots & \alpha^N
\end{array} \right)
\end{displaymath}
such that the maximal minors are relatively prime, we consider 
the $(N\times N-n)$-matrix $B$ such that $AB=0$.
Moreover, $B$ should be such that the rows in $B$ span $\Z^{N-n}$ and such that it is on the form $(B^{\prime}, E_m)^{tr}$ where 
$E_m$ is the unit $(N-n\times N-n)$-matrix. Let $\C^A$ be the vector space consisting of vectors $(a_{\alpha})_{\alpha\in A}$ and write $a=(a_1,\ldots,a_N)$.
Let $b^1,\ldots,b^{N-n}$ be the columns in $B$.
The differential operators $\Box_{i}$ and $\mathcal{E}_i$ on $\C^A$ are defined by
\begin{equation}
\Box_{i} = \prod_{j:b_j^i>0}\left( \partial/\partial a_j\right)^{b_j^i} - \prod_{j:b_j^i<0}\left( \partial/\partial a_j\right)^{-b_j^i}
\end{equation}
and
\begin{equation}
\mathcal{E}_i = \sum_{j=1}^{N}\alpha_i^{j}a_j(\partial/\partial a_j),\quad i=1,\ldots,n,
\end{equation}
where $\alpha^{i}_j$ is the entry in $A$ on row $i$ and column $j$.

\begin{definition}
For every complex vector $\gamma=(\gamma_1,\ldots,\gamma_n)$, we define the GKZ-system with parameters $\gamma$ as the following system of linear differential equations on functions $\Phi$ on $\C^A$.
\begin{equation}\label{gkzsys}
\Box_i\Phi(a)=0,\qquad \mathcal{E}_j\Phi=\gamma_j\Phi,\quad i=1,\ldots, N-n,\quad j=1,\ldots,n.
\end{equation}
\end{definition}
\noindent
The holomorphic solutions to the system (\ref{gkzsys}) are called $A$-hypergeometric functions.
A formal explicit solution to the system (\ref{gkzsys}) is given by 
\begin{equation}\label{Ahyp}
\Phi(a)=\sum_{k\in\Z^{N-n}}\frac{a^{\gamma + \left< B,k \right>}}{\prod_{j=1}^{n}\Gamma \left( \gamma_j+\left< B_j,k \right> +1\right) k!},
\end{equation}
where $B_j$ denotes the rows in the matrix $B$ and $\gamma_{n+1},\ldots,\gamma_N=0$.\\
\smallskip
Remember the formula 
\begin{equation}\label{gammsin}
\Gamma(s)\Gamma(1-s) = \pi/\sin(\pi s).
\end{equation}
In the generic case (noninteger parameters) the formula (\ref{gammsin}) directly gives us the following formula making it possible to move the gamma functions in (\ref{Ahyp}) from the 
denominator to the numerator, i.e.,
\begin{equation}\label{gammarel}
\frac{\Gamma(s+n)}{\Gamma(s)}=(-1)^n\frac{\Gamma(1-s)}{\Gamma(1-n-s)}.
\end{equation}
\noindent
We can now relate the functions $\gauss$ and $\Phi$ by
\begin{align*}
\gauss(a,b;c;z) &= \sum_{n=0}^{\infty}\frac{(a)_n(b)_n}{(c)_n}\frac{z^n}{n!}=
\sum_{n=0}^{\infty}\frac{\Gamma(a+n)\Gamma(b+n)\Gamma(c)}{\Gamma(c+n)\Gamma(a)\Gamma(b)}\frac{z^n}{n!}\\
&=\sum_{n=0}^{\infty}\frac{\Gamma(1-a)\Gamma(1-b)\Gamma(c)}{\Gamma(1-n-a)\Gamma(1-n-b)\Gamma(c+n)}\frac{z^n}{n!} \\
&=\Gamma(1-a)\Gamma(1-b)\Gamma(c)\Phi(1,1,1,z),
\end{align*}
with
\begin{equation*}
\gamma = (-a,-b,c-1) \quad \text{and}\quad B=(-1,-1,1,1)^{tr}.
\end{equation*}
 The above equation together with (\ref{gaussAhyp}) make it possible for us to express the complete elliptic integral of the first kind as an 
$A$-hypergeometric function, i.e.,
\begin{align}\label{forstaphi}
\K(k) = \frac{\pi^2}{2}\Phi(1,1,1,z),
\end{align}
with
\begin{equation*}
\gamma = (-1/2,-1/2,0) \quad \text{and}\quad B=(-1,-1,1,1)^{tr}.
\end{equation*}

We can do the same procedure for the Appell hypergeometric function $\appellett$ but we have to modify the function $\Phi$ a bit because we have a non generic parameter in the numerator. We therefore introduce the series $\tilde{\Phi}$ defined by
\begin{equation}
\tilde{\Phi}(a) = \sum_{k\in\Z^{N-n}}\frac{(-1)^{\left< B_1,k \right>}\Gamma\left(-\gamma_1-\left< B_1,k \right>\right) a^{\gamma + \left< B,k \right>}}{\prod_{j=2}^{n}\Gamma \left( \gamma_j+\left< B_j,k \right> +1\right) k!}.
\end{equation}
\noindent
The series $\tilde{\Phi}$ should be regarded as a meromorphic function with removable singularities (the $k!$ in the denominator take care of the possible singularities of the gamma function in the numerator). Note that for generic parameters we can use (\ref{gammarel}) to move the gamma function in the numerator to the denominator and we get
\begin{equation*}
\tilde{\Phi} = \Gamma(1+\gamma_1)\Gamma(-\gamma_1)\Phi.
\end{equation*}
\noindent
We can now do the same reasoning as in the case of the Gauss hypergeometric function and use (\ref{tredjeochappell}) to get
\begin{align}\label{tredjehyper}
\Pi(\alpha^2,k) = \frac{\pi^2}{2}\tilde{\Phi}(1,1,1,1,\alpha^2,k^2),
\end{align}
with
\begin{equation*}
\gamma = (-1,0,-1/2,-1/2) \quad \text{and}\quad B = 
\left( \begin{array}{cccccc}
-1 & 1 & 0 & -1 & 1 & 0 \\
0 & 1 & -1 & -1 & 0 & 1
\end{array} \right)^{tr}.
\end{equation*}

\noindent
If we combine (\ref{forstaphi}) and (\ref{tredjehyper}) with Proposition \ref{andraderivator} we get an expression of the second order derivatives of the Ronkin function of an affine linear polynomial in three variables in terms of $A$-hypergeometric functions.

\begin{proposition}
Let $f=1+z+w+t$ and set
\begin{equation*}
 \gamma_1 = (-1/2,-1/2,0),\quad \gamma_2 = (-1,0,-1/2,-1/2),
\end{equation*}
\begin{equation*}
\quad B=(-1,-1,1,1)^{tr},\quad B_2= 
\left( \begin{array}{cccccc}
-1 & 1 & 0 & -1 & 1 & 0 \\
0 & 1 & -1 & -1 & 0 & 1
\end{array} \right)^{tr}.
\end{equation*}
Then the second order derivatives of the Ronkin function $N_f$ can be expressed in terms of $A$-hypergeometric functions in the following way.
\begin{align*}
\frac{\partial^2 N_f}{\partial x^2}  & = g e^{2x} \Phi(1,1,1,k^2), \\
\frac{\partial^2 N_f}{\partial x \partial y}  & = \frac{g}{4}( Q_1 \Phi(1,1,1,k^2) + Q_2 \tilde{\Phi}(1,1,1,1,\alpha_1^2,k^2)
 + Q_3 \tilde{\Phi}(1,1,1,1,\alpha_2^2,k^2) )
\end{align*}
with parameters $\gamma_1,\gamma_2$ and matrices $B_1,B_2$. The functions and parameters $k^2$, $\alpha_1^2, \alpha_2^2, g^2, Q_1, Q_2$ and $Q_3$ are defined in Proposition \ref{andraderivator}.
\end{proposition}


\section{The logarithmic Mahler measure}\label{Mahler}

Closely related to the Ronkin function is the Mahler measure that was introduced by Mahler in \cite{Mahler}.
The Mahler measure of a polynomial is a real number and the logarithm of that number is called the logarithmic mahler measure.
\begin{definition}
Let $f$ be a polynomial in $n$ variables with real or complex coefficients. The number
\begin{equation*}
\m(f) = 
\left\{ \begin{array}{ll}
\left( \frac{1}{2\pi i}\right)^n\int_{\Log^{-1}(0)}\log|f(z)|\frac{dz}{z}& \text{if}\quad f \not\equiv 0\\
0& \text{if}\quad f \equiv 0
\end{array} \right.
\end{equation*}
is called the {\rm logarithmic Mahler measure} of $f$.
\end{definition}

\noindent
We see that the logarithmic Mahler measure is the Ronkin function evaluated in the origin. On the other hand if 
$f(z)=\sum_{\alpha\in A}a_{\alpha}z^{\alpha}$, then 
\begin{equation*}
N_{f}(x) = \m\left( \sum_{\alpha\in A}a_{\alpha}e^{\left< \alpha, x \right>}z^{\alpha}\right).
\end{equation*}
In particular, if $f(z_1,\ldots,z_n)=1+z_1+\ldots+z_n$ we get that 
\begin{equation}\label{RonMah}
N_f(x_1,\ldots,x_n)=\m(1+e^{x_1}z_1+\ldots+e^{x_n}z_n).
\end{equation}
Thus if one can give an explicit expression of the Mahler measure of $f=1+a_1z_1+\ldots+a_nz_n$ for $a_j>0$ one also has an explicit expression of the Ronkin function of 
$f=1+z_1+\ldots+z_n$ and vice versa.

One of the first explicit formulas for the Mahler measure of a two variable polynomal was proved by Smyth, \cite{Smyth}, and takes the following form in terms of the Ronkin function.

\begin{theorem}{\rm (Smyth)}
Let $f=1+z+w$. Then
\begin{equation*}
N_f(0,0) = \frac{3\sqrt{3}}{4\pi}\operatorname{L}(\chi_{-3},2),
\end{equation*}
where 
\begin{equation*}
\operatorname{L}(\chi_{-3},s)=\sum_{k=1}^{\infty}\frac{\chi_{-3}(k)}{k^s} \quad\text{and}\quad
\chi_{-3}(k)=
\left\{ \begin{array}{ll}
1& \text{if}\quad k\equiv 1 \mod 3\\
-1& \text{if}\quad k\equiv -1 \mod 3\\
0& \text{if}\quad k\equiv 0 \mod 3
\end{array} \right..
\end{equation*}
\end{theorem}
\noindent
Almost 20 years later Maillot generalized the theorem of Smyth by giving an explicit expression for the Ronkin function at every point in $\R^2$, see \cite{Maillot}. The expression involves the so-called Block-Wigner dilogarithm, denoted by $D(z)$ and defined as

\begin{equation*}
D(z)=\operatorname{Im}(\Li_2(z)+\log|z|\log(1-z))
\end{equation*}
for $z\in\mathbb{C}^n\setminus \{0,1\}$. Here $\Li_2(z)$ is the dilogarithm of $z$.

\begin{theorem}{\rm (Maillot)}
Let $f = 1+z+w$. Then
\begin{equation*}
N_f(x,y) = 
\left\{ \begin{array}{lll}
\frac{\alpha}{\pi}x+\frac{\beta}{\pi}y + \frac{1}{\pi}D(e^{x+i\beta}) & \text{if } (x,y)\in\Af\\
 \\
\max\{ 0, x, y \} & \text{otherwise}
\end{array} \right.,
\end{equation*}
where $\alpha$ and $\beta$ are defined in Figure \ref{Triangel2} below.
\end{theorem}

\noindent
Interestingly, the partial derivatives of the Ronkin function are very easy to describe in this case.
\begin{example}{\rm
Let $f(z,w)=1+z+w$. Then
\begin{equation*}
\ddx N_f  = \frac{\alpha}{\pi}, \qquad \ddy N_f = \frac{\beta}{\pi},
\end{equation*} 
where $\alpha$ and $\beta$ are described in Figure \ref{Triangel2}.

\begin{figure}[h]
\begin{center}
\includegraphics[height=3cm]{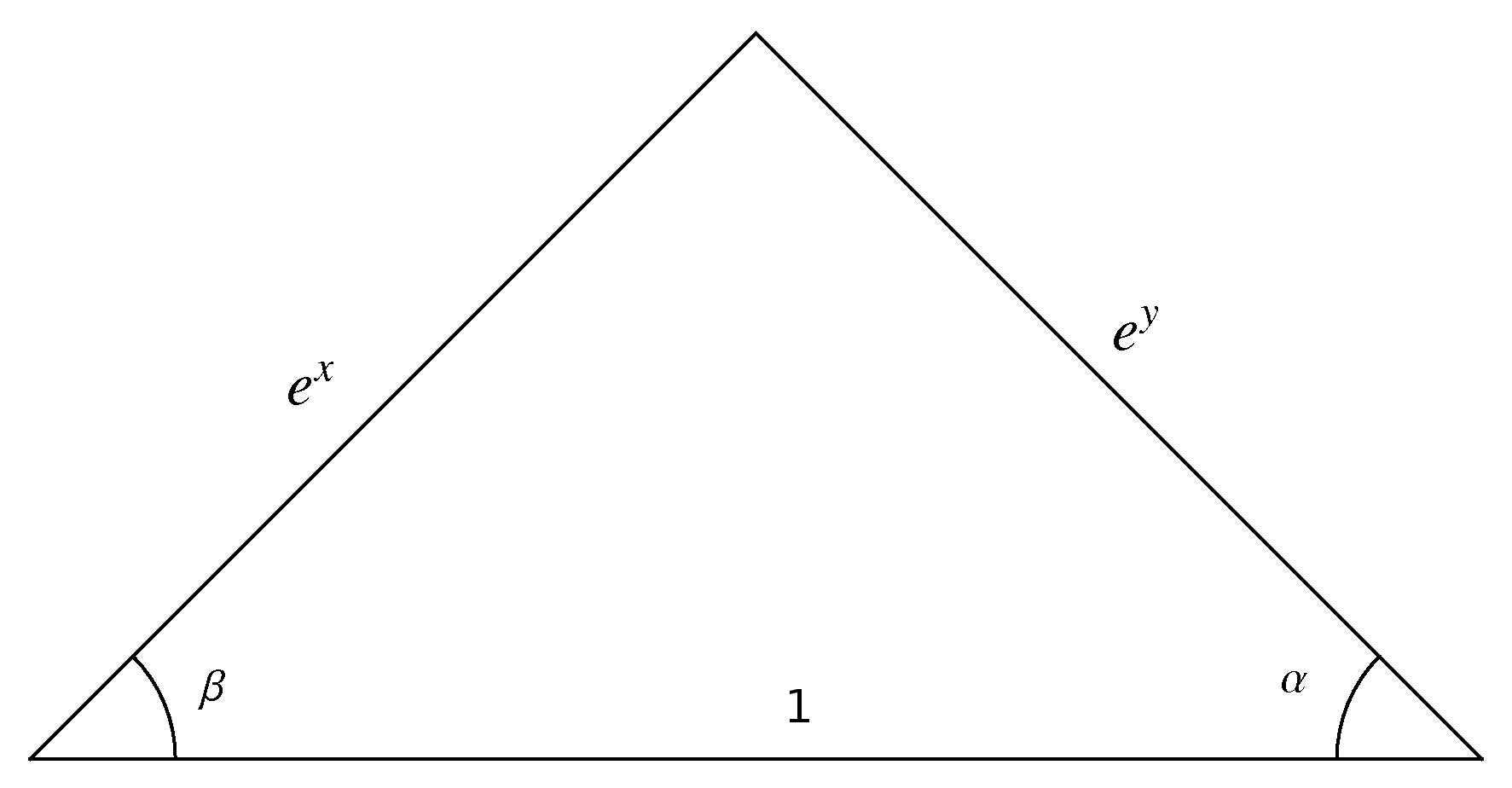}
\caption{}\label{Triangel2}
\label{Triangel}
\end{center}
\end{figure}

To see this, note that
a differentiation under the integral sign gives
\begin{eqnarray*}
\ddx N_f(x,y) &=& \ddx \left( \frac{1}{2\pi i} \right)^2\int_{\Log^{-1}(x,y)}\log|1+z+w|\frac{dz}{z}\frac{dw}{w} \\
&=& \left( \frac{1}{2\pi i} \right)^2\int_{\Log^{-1}(x,y)}\frac{dz}{(1+z+w)}\frac{dw}{w} \\
&=& \left( \frac{1}{2\pi i} \right) \int_{|w|=e^y} \left( \left( \frac{1}{2\pi i} \right) \int_{|z|=e^x}\frac{dz}{z-(-1-w)}\right) \frac{dw}{w}.
\end{eqnarray*}
Now, the inner integral is equal to $1$ when $|z|=e^x<|1+w|$ and equal to $0$ when $|z|=e^x>|1+w|$. Since $dw/w$ is the volume measure on the torus $|w|=e^y$ we get that 
$N_f$ equals the ratio 
\begin{equation*}
\frac{\lambda \left( \{ \phi\in[0,2\pi ]; e^x < |1+e^{y+i\phi}| \} \right)}{\lambda \left( [ 0,2\pi ]\right)},
\end{equation*}
where $\lambda$ is the Lebesgue measure, and this expression is obviously equal to $\alpha/\pi$. The second part is proved analogously.}
\end{example}

In \cite{Smyth} Smyth  proved a formula for the affine linear case in the three variables case but this only gives 
the values of the Ronkin function at points where four of the chambers meet.
\begin{theorem}{\rm (Smyth)}\label{smy}
\begin{displaymath}
m(1+z+aw+at)  = \left\{ \begin{array}{ll}
\frac{2}{\pi^2}\left( \operatorname{Li}_3(a)-\operatorname{Li}_3(-a) \right) & \textrm{if } a\leq 1\\
\log(a) + \frac{2}{\pi^2}\left( \operatorname{Li}_3(a^{-1})-\operatorname{Li}_3(-a^{-1}) \right)& \textrm{if }a \geq 1
\end{array} \right.,
\end{displaymath}
where $\operatorname{Li_3}$ is the trilogarithm defined as
\begin{equation*}
\operatorname{Li_3}(z) = \sum_{k=1}^{\infty}\frac{z^k}{k^3}.
\end{equation*}
\end{theorem}
No more general formula has been proved so far.
Note that the theorem by Smyth and formula \eqref{deriv} give us the formula
\begin{align*}
& \Li_2(-e^x)-\Li_2(e^x) 
 =  \int_{1-e^x}^{1+e^x}\arccos\left( \frac{1+r^2-e^{2x}}{2r}\right) \frac{d}{dr}\arccos\left( \frac{r^2-1-e^{2x}}{2e^x}\right) dr,
\end{align*}
for $e^x<1$.
Maybe there is a similar kind of relation in the more general expression of \eqref{deriv}?

It seems to be of interest to estimate affine linear polynomials in $n$ variables, both for fixed $n$ or when $n$ tends to infinity.
In \cite{Toledano} the author proves that there exists an analytic function $F$ such that the Mahler measure of the linear form $z_1+\ldots+z_n$ up to an explicit constant is equal to $F(1/n)$. There is also an recursive expression of that analytic function in terms of Laguerre polynomials and Bessel functions. Note that this corresponds to the Ronkin function evaluated at the origin. In the paper \cite{Rod} the authors estimate the growth of the Mahler measure in the linear case when the number of variables goes to infinity and also establish a lower and upper bound in terms of the norm of the coefficient vector. The reason for the interest in these kind of estimates is that it is hard to calculate the Mahler measure numerically and numerical calculations are of interest when looking for relations between the Mahler measure and special values of $L$-functions.
Several such relations has been conjectured by Boyd, see \cite{Boyd2}.
We have not calculated the actual Ronkin function but all the second order derivatives. Note that the Ronkin function of $f=1+z+w+t$ is determined by its second order derivatives up to a polynomial on the form $a+b(x+y+u)$.

\end{document}